\documentclass[12pt,a4paper]{article}
\usepackage{amssymb,amsmath,amsthm}
\usepackage{graphicx}
\usepackage{hyperref}
\usepackage{xcolor}

\newcommand\cA{{\mathcal A}}

\newcommand\cF{{\mathcal F}}

\newcommand\ex{\ensuremath{\mathrm{ex}}}
\newcommand\exo{\ensuremath{\mathrm{ex_o}}}
\newcommand\exd{\ensuremath{\mathrm{ex_d}}}

\theoremstyle{plain}
\newtheorem{theorem}{Theorem}[section]
\newtheorem{lemma}[theorem]{Lemma}
\newtheorem{corollary}[theorem]{Corollary}

\newtheorem{proposition}[theorem]{Proposition}

\theoremstyle{definition}

\newtheorem{definition}[theorem]{Definition}

\newtheorem{claim}[theorem]{Claim}

\newtheorem*{lemma*}{Lemma}
\newtheorem*{thm*}{Theorem}

\newcommand\cref[1]{Corollary~\ref{cor:#1}}

\title{On oriented Tur\'an problems}

\author{D\'aniel Gerbner\thanks{HUN-REN Alfr\'ed R\'enyi Institute of Mathematics, E-mail: \texttt{gerbner@renyi.hu}. Research supported by the National Research, Development and Innovation Office - NKFIH under the grant KKP-133819 and by the J\'anos Bolyai scholarship.} , Xuanrui Hu\thanks{School of Mathematics and Statistics, Ningbo University, Ningbo 315211, China. Email: \texttt{huxuanrui1108@163.com.}} , Yuefang Sun\thanks{Corresponding author. School of Mathematics and Statistics, Ningbo University, Ningbo 315211, China. Email: \texttt{sunyuefang@nbu.edu.cn}. Research supported by National Natural Science Foundation of China under Grant No. 12371352 and Yongjiang Talent Introduction Programme of Ningbo under Grant No. 2021B-011-G.}}
\date{}
\begin{document}

\maketitle
\begin{abstract}
The oriented Tur\'{a}n number of a given oriented graph $\overrightarrow{F}$, denoted by $\exo(n,\overrightarrow{F})$, is the largest number of arcs in $n$-vertex $\overrightarrow{F}$-free oriented graphs. This concept could be seen as an oriented version of the classical Tur\'{a}n number. 

In this paper, we first prove several propositions that give exact results for several oriented graphs. In particular, we determine all exact values of $\exo(n,\overrightarrow{F})$ for every oriented graph $\overrightarrow{F}$ with at most three arcs and sufficiently large $n$. After that, we prove a stability result and use it to determine the Turán number of an orientation of $C_4$. Finally, we prove oriented versions of the random zooming theorem by  Fern\'andez, Hyde, Liu, Pikhurko and Wu and the almost regular subgraph theorem by Erd\H{o}s and Simonovits, and use them to obtain an oriented version of the F\"{u}redi-Alon-Krivelevich-Sudakov Theorem, which generalizes the famous KST Theorem. 

\end{abstract}

\vspace{0.3cm}

{\bf Keywords}: Oriented graph; oriented Tur\'{a}n number; compressibility

\section{Introduction}

One of the most fundamental theorems in extremal graph theory is due to
Turán \cite{T1941}, who determined the largest number of edges in a $K_k$-free graph on $n$ vertices. More generally, the Turán number $\ex(n,F)$ is the largest number of edges in $n$-vertex graphs that do not contain $F$ as a not necessarily induced subgraph. The celebrated 
Erdős-Stone-Simonovits \cite{ES1946,ES1966} states that for any graph $F$, we have $\ex(n,F)=\left(\frac{\chi(F)-2}{\chi(F)-1}+o(1)\right)\binom{n}{2}$, which determines the asymptotics of the Turán number for non-bipartite graphs.

In this paper we study analogous problems for oriented graphs.
Given an oriented graph $\overrightarrow{F}$, we study the largest number of arcs in $n$-vertex oriented graphs that do not contain $\overrightarrow{F}$ as a subgraph. This quantity is denoted by $\exo(n,\overrightarrow{F})$, and its study was initiated by Valadkhan \cite{vala}, who proved an analogue of the Erdős-Stone-Simonovits theorem. We will also consider the quantity $\exd(n,\overrightarrow{F})$, which is the largest number of arcs in $n$-vertex \emph{directed} graphs that do not contain $\overrightarrow{F}$ as a subgraph. Clearly, $\exo(n,\overrightarrow{F})\le \exd(n,\overrightarrow{F})\le 2\exo(n,\overrightarrow{F})$. Let us remark that $\exd(n,\overrightarrow{F})$ has been more widely studied, first by Brown and Harary \cite{bh}.

 Let $\overrightarrow{T_k}$ denote
the transitive
tournament on $k$ vertices. Let
$\overrightarrow{P_k}$ be the directed path on $k$ vertices, where every arc is directed towards the same endpoint of the path. Similarly $\overrightarrow{C_k}$ denotes the directed cycle, where each vertex has an incoming and an outgoing arc. Let $\overrightarrow{S_{p,q}}$ be the orientation of a star on $p + q + 1$ vertices with center vertex of in-degree $p$ and out-degree $q$.

A {\em homomorphism} from a digraph $H$ to another digraph $D$ is a mapping $f:V(H)\rightarrow V(D)$ such that $f(u)f(v)\in E(D)$ whenever $uv\in E(H)$. The \textit{compressibility} $z(\overrightarrow{F})$ of an oriented graph $\overrightarrow{F}$ is the smallest $k$ such that there is a homomorphism from $\overrightarrow{F}$ to any tournament of order $k$.

\begin{theorem}[Valadkhan \cite{vala}]\label{ValadkhanTHM} For any acyclic oriented graph $\overrightarrow{F}$ we have 

    \[\exo(n,\overrightarrow{F})=\left(\frac{z(\overrightarrow{F})-2}{z(\overrightarrow{F})-1}+o(1)\right)\binom{n}{2}.\]
\end{theorem}

The above theorem determines the asymptotics of the Turán number for most oriented graphs. If $\overrightarrow{F}$ contains a directed cycle, then the transitive tournament is $\overrightarrow{F}$-free, thus $\exo(n,\overrightarrow{F})=\binom{n}{2}$. If $z(\overrightarrow{F})=2$, then the above theorem implies only $\exo(n,\overrightarrow{F})=o(n^2)$. Note that $z(\overrightarrow{F})=2$ if and only if $\overrightarrow{F}$ is \textit{antidirected}, i.e., contains only sources and sinks.

The rest of \cite{vala} is dedicated to giving bounds on the compressibility. Another paper \cite{gjkns} deals with this topic and addresses this topic, also focusing on bounding the compressibility. Further results have been obtained in the case of antidirected trees in the directed setting that are interesting for us.

Graham \cite{gra} showed that for each antidirected tree $\overrightarrow{F}$ with $k$ arcs, $\exd(n,\overrightarrow{F})=O(n)$. The constant factor was improved later, to $\exd(n,\overrightarrow{F})\le 4kn$ by Burr \cite{burr}. Addario-Berry, Havet, Linhares Sales, Reed and Thomass\'e \cite{ahlrt} conjectured that $\exd(n,\overrightarrow{F})\le (k-1)n$ and proved this for trees of diameter at most 3. Grzesik and Skrzypczyk \cite{gs} showed that for the  antidirected path $\overrightarrow{F}$ with $k$ arcs, $\exo(n,\overrightarrow{F})\le (k-1+\sqrt{k-3})n$.

Besides trees, we are not aware of any results that goes beyond determining the compressibility. In particular, no exact results are obtained. In this paper, we obtain exact results and also sharp upper bounds for antidirected graphs.

\subsection{Abstract chromatic number}

Compressibility fits into a recent framework introduced by Gerbner, Hama Karim and Kucheriya \cite{ghk}, based on the work of Coregliano and Razborov \cite{CoregRazb}. Let $\cA(\overrightarrow{F})$ denote the family of unoriented graphs with an orientation that avoids $\overrightarrow{F}$, and $\cF(\overrightarrow{F})$ denote the family of unoriented graphs that contain $\overrightarrow{F}$ in every orientation. 

Clearly, if $\overrightarrow{G}$ is $\overrightarrow{F}$-free, then its underlying graph $G$ is in $\cA(\overrightarrow{F})$, and for every $G\in\cA(\overrightarrow{F})$ there is an $\overrightarrow{F}$-free orientation. Therefore, $\exo(n,\overrightarrow{F})=\max\{|E(G)|: G\in\cA(\overrightarrow{F}), \, |V(G)|=n\}$.

The \textit{abstract chromatic number} of the partition  $(\cA(\overrightarrow{F}),\cF(\overrightarrow{F})$ is defined to be the largest $k$ such that every complete $(k-1)$-partite graph with each part of order at least $n$ is in $\cA$, for every sufficiently large $n$, or $\infty$ if $K_n\in\cA$ for all sufficiently large $n$. Therefore, the abstract chromatic number is the compressibility of $\overrightarrow{F}$ if $\overrightarrow{F}$ is acyclic, and $\infty$ if $\overrightarrow{F}$ contains a directed cycle.

The main advantage of this framework is that these results extend to some cases when we want to maximize some graph parameter other than edges, for example the number of some other (unoriented) subgraphs, or the spectral radius or some topological indices of the underlying unoriented graph, see \cite{ghk} for more details. 

Another advantage, more relevant to us, is that some other results concerning Turán numbers also extend to our setting. In particular, the following stability theorem holds by applying Theorem 2.1 from \cite{ghk} in our setting.

\begin{theorem}\label{stabi}
    If $\overrightarrow{G}$ is an $\overrightarrow{F}$-free oriented graph with at least $\left(\frac{z(\overrightarrow{F})-2}{z(\overrightarrow{F})-1}-o(1)\right)\binom{n}{2}$ arcs, then we can turn the underlying graph $G$ to the Tur\'an graph $T(n,z(\overrightarrow{F})-1)$ by adding and deleting $o(n^2)$ edges.
\end{theorem}

As it is often the case in extremal Combinatorics, stability will help us obtain exact results.
Let us remark that for us, similarly useful from this framework is the approach that we look at the underlying unoriented graphs. 

\subsection{Our results}

The following result by F\"{u}redi~\cite{Furedi}, and Alon, Krivelevich and Sudakov~\cite{Alon-Krivelevich-Sudakov} generalizes the famous KST Theorem \cite{kst}:

\begin{theorem}\textbf{(F\"{u}redi-Alon-Krivelevich-Sudakov Theorem)}\label{FAKS}
Let $H$ be a bipartite graph with bipartition $A\cup B$ where each vertex in $A$ has degree at most $r$ in $B$. Then there exists a constant $c=c(H)$ such that $$\ex(n,H)\le c\cdot n^{2-\frac{1}{r}}.$$    
\end{theorem}
Our main theorem is the following oriented strengthening.

\begin{theorem}\label{main}
Let $H=H(A\cup B,A\rightarrow B)$ be an $h$-vertex bipartite digraph such that $|N_B^+(a)|\leq r$ for each $a\in A$. Then there exists a constant $c=c(H)$ such that $$\exo(n,H)\le c\cdot n^{2-\frac{1}{r}}.$$
\end{theorem}

In Section 2, we prove several propositions that give exact results for several oriented graphs. In particular, we determine all exact values of $\exo(n,\overrightarrow{F})$ for every oriented graph $\overrightarrow{F}$ with at most three arcs and sufficiently large $n$. In Section 3, we prove a strengthening of Theorem \ref{stabi} for some graphs and use it to determine the Turán number of an orientation of $C_4$. We prove Theorem \ref{main} in Section 4. We finish the paper with some concluding remarks in Section 5.

\section{Basic propositions}


\begin{proposition}\label{klikk}
    If any tournament on $z(\overrightarrow{F})$ vertices contains $\overrightarrow{F}$ as a subgraph, then $\exo(n,\overrightarrow{F})=|E(T(n,z(\overrightarrow{F})-1))|$.
\end{proposition}

\begin{proof}
    Clearly, $T(n,z(\overrightarrow{F})-1)$ has an orientation that avoids $\overrightarrow{F}$, giving the lower bound. Let $\overrightarrow{G}$ be an $\overrightarrow{F}$-free oriented graph, then the underlying graph $G$ must be $K_{z(\overrightarrow{F})}$-free, since a $K_{z(\overrightarrow{F})}$ in $G$ is a tournament on $z(\overrightarrow{F})$ vertices in $\overrightarrow{G}$. Tur\'an's theorem \cite{T1941} completes the proof.
\end{proof}

\begin{corollary}\label{coroll}
     \textbf{(i)} $\exo(n,\overrightarrow{P_k})=|E(T(n,k-1))|$.
     
     \textbf{(ii)} For any tournament $\overrightarrow{F}$ we have $\exo(n,\overrightarrow{F})=|E(T(n,z(\overrightarrow{F})-1))|$.

     \textbf{(iii)} Let $\overrightarrow{F}$ be the oriented $C_4$ with arcs $\overrightarrow{v_1v_2}$, $\overrightarrow{v_2v_3}$, $\overrightarrow{v_3v_4}$, $\overrightarrow{v_1v_4}$. Then $\exo(n,\overrightarrow{F})=|E(T(n,3)|$.
\end{corollary}

We remark that Zhou and Li \cite{zl} showed that $\exd(n,\overrightarrow{P_k})=|E(T(n,k-1))|$ for sufficiently large $n$.

\begin{proof}
    As pointed out in \cite{gjkns}, $\overrightarrow{P_k}$ has no homomorphism into the transitive tournament on $k-1$ vertices, and by a theorem of Rédei \cite{red}, every tournament on $k$ vertices contains $\overrightarrow{P_k}$.

    Clearly, a homomorphism from a tournament into a loopless graph cannot map two vertices to the same vertex, and thus must be an isomorphism.

    To prove \textbf{(iii)}, observe that $\overrightarrow{F}$ contains $\overrightarrow{P_4}$, thus $z(\overrightarrow{F})\ge 4$. Consider now a tournament $\overrightarrow{T}$ on 4 vertices, by the above, we can assume that $\overrightarrow{v_1v_2}$, $\overrightarrow{v_2v_3}$, $\overrightarrow{v_3v_4}$ are in $\overrightarrow{T}$. If $\overrightarrow{v_1v_4}$ is also in $\overrightarrow{T}$, we are done, thus we can assume that $\overrightarrow{v_4v_1}$ is in $\overrightarrow{T}$, i.e., we have a $C_4$ with arcs oriented cyclically. Then, without loss of generality we have $\overrightarrow{v_1v_3}$ in $\overrightarrow{T}$. If $\overrightarrow{v_4v_2}$ is in $\overrightarrow{T}$, then $\overrightarrow{v_1v_3}$, $\overrightarrow{v_3v_4}$ and $\overrightarrow{v_4v_2}$ together with $\overrightarrow{v_1v_2}$ forms $\overrightarrow{F}$. Finally, if $\overrightarrow{v_2v_4}$ is in $\overrightarrow{T}$, then $\overrightarrow{v_2v_4}$, $\overrightarrow{v_4v_1}$ and $\overrightarrow{v_1v_3}$ together with $\overrightarrow{v_2v_3}$ forms an $\overrightarrow{F}$. Therefore, $\overrightarrow{T}$ contains $\overrightarrow{F}$, hence Proposition \ref{klikk} completes the proof.
\end{proof}

We say that an edge of a graph is \textit{color-critical} if deleting it decreases the chromatic number of the graph. A theorem of Simonovits \cite{sim2} states that if $F$ has chromatic number $k$ and a color-critical edge, then for sufficiently large $n$ we have $\ex(n,F)=|E(T(n,k-1))|$. We can extend the idea of Proposition \ref{klikk} in the case every orientation a graph of chromatic number $z(\overrightarrow{F})$ with a color-critical edge contains $\overrightarrow{F}$. The following result is an example of this.

\begin{proposition}\label{pro2.3}
Let $\overrightarrow{F}$ denote the 4-vertex oriented graph with arcs $\overrightarrow{v_1v_2},\overrightarrow{v_2v_3},\overrightarrow{v_4v_3}$. Then for sufficiently large $n$, $\exo(n,\overrightarrow{F})=\lfloor n^2/4\rfloor$.
\end{proposition}

\begin{proof}
The lower bound is given by the antidirected orientation on $T(n,2)$.

    Let $H$ denote the following graph. We take $K_{2,4}$ and add an edge $uv$ inside the part of size 4. Note that the additional edge is color-critical. We claim that every orientation of $H$ contains $\overrightarrow{F}$. Therefore, the underlying graph of an $\overrightarrow{F}$-free oriented graph is $H$-free, hence contains at most $\ex(n,H)$ edges, which is at most $\lfloor n^2/4\rfloor$ if $n$ is sufficiently large by the theorem of Simonovits \cite{sim2} mentioned above. 

    Let $x,y$ be the vertices in the part of size 2 in $H$ and $w$ be a vertex distinct from $u,v,x,y$. Consider the cycle $xuyvx$. It avoids $\overrightarrow{F}$ if and only if it is a directed cycle $\overrightarrow{C_4}$ or an antidirected cycle. If it is a directed cycle which contains, say, the arc $\overrightarrow{vx}$, then $xw$ and $yw$ are oriented towards $w$. Then $\overrightarrow{vx},\overrightarrow{xw},\overrightarrow{yw}$ form a copy of $\overrightarrow{F}$.

    Therefore, we can assume that $xuyvx$ is an antidirected cycle. In other words, these four vertices form $\overrightarrow{K_{2,2}}$. If $u$ and $v$ are the sources here, then no matter what the orientation of $uv$ is, $x$ and $y$ are both endpoints of copies of $\overrightarrow{P_3}$, thus the edges $xw$ and $yw$ are oriented towards $w$ and we find a copy of $\overrightarrow{F}$ as in the previous paragraph. Finally, if $u$ and $v$ are the sinks in the $\overrightarrow{K_{2,2}}$, then without loss of generality $\overrightarrow{uv}$ is an arc, and then $\overrightarrow{xu},\overrightarrow{uv},\overrightarrow{yv}$ form a copy of $\overrightarrow{F}$. This completes the proof.
\end{proof}

Let $\overrightarrow{F}$ denote the 4-vertex oriented graph with arcs $\overrightarrow{v_2v_1},\overrightarrow{v_2v_3},\overrightarrow{v_3v_4}$. We still have that $\exo(n,\overrightarrow{F})=\lfloor n^2/4\rfloor$  for sufficiently large $n$, as this case is equivalent to that in Proposition~\ref{pro2.3} by symmetry: we just turn every edge and relabel the vertices.


\begin{theorem}\label{star}
    \textbf{(i)} $\exo(n,\overrightarrow{S_{0,q}})=(q-1)n$.
    
    \textbf{(ii)} If $1\le p\le q$, then for sufficiently large $n$ we have $\exo(n,\overrightarrow{S_{p,q}})=(p-1)n+\left\lfloor\frac{(n+q-p)^{2}}{4}\right\rfloor$.
\end{theorem}

\begin{proof}
    The upper bound in \textbf{(i)} is immediate, since at most $q-1$ arcs go in every vertex. The lower bound is obtained by the $(q-1)$st power of the cycle, i.e., we have vertices $v_1,\dots, v_n$, and $\overrightarrow{v_iv_j}$ is an arc if and only if $j-1$ is between 1 and $q-1$ modulo $i$.

    For \textbf{(ii)}, observe first that $z(\overrightarrow{S_{p,q}})=3$. Let $\overrightarrow{G}$ be an $\overrightarrow{S_{p,q}}$-free $n$-vertex graph with $\exo(n,\overrightarrow{S_{p,q}})$ arcs.
    Let $C$ denote the set of vertices with outdegree at least $q$ and $D$ denote the set of vertices with indegree at least $p$. Clearly $C$ and $D$ are disjoint, without loss of generality $|C|\le |D|$. The sum of the indegrees is $\exo(n,\overrightarrow{S_{p,q}})\ge \lfloor n^2/4\rfloor$, since the antidirected $T(n,2)$ is $\overrightarrow{S_{p,q}}$-free. 
As vertices not in $D$ have total indegree at most $(n-|D|)p$, the sum of the indegrees is at most $n|D|+O(n)$. Therefore, $|D|\ge n/4-o(n)$. If a vertex $u$ is not in $C\cup D$, then it is incident to at most $p+q-2$ arcs. We delete all the incident edges and add all the arcs $\overrightarrow{uv}$ for all $v\in D$. Then neither $u$, nor $v$ becomes a center of an $\overrightarrow{S_{p,q}}$ and the number of arcs increases, a contradiction.

    We obtained that $V(\overrightarrow{G})$ is partitioned to $C$ and $D$. Each vertex in $C$ has indegree at most $p-1$, thus there are at most $(p-1)|C|$ arcs inside $C$. Analogously, there are at most $(q-1)|D|$ arcs inside $D$, and clearly there are at most $|C||D|$ arcs between $C$ and $D$. Observe that we can have a construction matching this upper bound the following way. Let $u_1,\dots, u_{|C|}$ be the vertices of $C$ and $w_1,\dots, w_{|D|}$ be the vertices of $D$. We take each arc from $u_i$ to the next $p-1$ vertices $u_{i+1}, \dots, u_{i+p-1}$ modulo $|C|$, and we take analogously each arc from $w_j$ to the next $q-1$ vertices $w_{j+1}, \dots, w_{j+q-1}$ modulo $|D|$. Finally, we take each arc of the form $\overrightarrow{u_iw_j}$. Then clearly we have $(p-1)|C|$ arcs inside $C$, $(q-1)|D|$ arcs inside $D$, and $|C||D|$ arcs between $C$ and $D$. Each vertex has either in-degree at most $p-1$ or out-degree at most $q-1$, thus cannot be the center of an $\overrightarrow{S_{p,q}}$.

    It is left to pick the size of $C$ and $D$. Our bound is equal to $(p-1)n+|D|(n-|D|+q-p)$. This is maximal if $|D|=n-|D|+q-p$, completing the proof.
\end{proof}

\begin{proposition}
    Let $\overrightarrow{F}$ be the vertex-disjoint union of $\overrightarrow{F_1}$ and $\overrightarrow{F_2}$ with $z(\overrightarrow{F_1})<z(\overrightarrow{F_2})$. Then for sufficiently large $n$, we have $\exo(n,\overrightarrow{F})=\exo(n,\overrightarrow{F_2})$.
\end{proposition}

\begin{proof}
    Clearly  $\exo(n,\overrightarrow{F})\ge\exo(n,\overrightarrow{F_2})$. To see the other direction, assume that $\overrightarrow{G}$ is an $n$-vertex graph with more than $\exo(n,\overrightarrow{F_2})$ arcs. Then it contains a copy of $\overrightarrow{F_2}$. Deleting the vertices of that copy, we obtain a graph $\overrightarrow{G'}$ on $n-|V(\overrightarrow{F_2})|$ vertices and at least $\exo(n,\overrightarrow{F_2})-O(n)$ arcs. Observe that $\exo(n,\overrightarrow{F_2})\ge\exo(n,\overrightarrow{F_1})+\Theta(n^2)$ since $z(\overrightarrow{F_1})<z(\overrightarrow{F_2})$. Therefore, $\exo(n,\overrightarrow{F_2})-O(n)>\exo(n,\overrightarrow{F_1})\ge \exo(n-|V(\overrightarrow{F_2})|,\overrightarrow{F_1})$, thus
    we find a copy of $\overrightarrow{F_1}$ in $\overrightarrow{G'}$, hence a copy of $\overrightarrow{F}$ in $\overrightarrow{G}$.
\end{proof}


We note that all the orientations of a matching $M_k$ give isomorphic oriented graphs $\overrightarrow{M_k}$. Therefore, $\exo(n,\overrightarrow{M_k})=\ex(n,M_k)=(k-1)(n-k+1)+\binom{k-1}{2}$, using a theorem of Erd\H os and Gallai \cite{EGa}.

\begin{proposition}
    Let $\overrightarrow{F}$ be the antidirected 4-vertex path and $n>1$. Then $\exo(n,\overrightarrow{F})=2n-3$. 
\end{proposition}

\begin{proof}
For the lower bound, consider an oriented graph with an arc $\overrightarrow{vu}$ such that $\overrightarrow{uw}$ and $\overrightarrow{wv}$ are arcs for each of the other $n-2$ vertices $w$. Then such a $w$ cannot be a middle vertex of an antidirected $P_4$, thus $vu$ must be the middle arc, but then another arc of this path is $uw$, thus it is not an antidirected path, a contradiction.

For the upper bound, the statement is trivial for $n\le 3$. let $A$ denote the set of vertices with one outgoing arc and $B$ denote the set of vertices with more outgoing arcs. Let $A'$ denote the set of vertices that are out-neighbors of vertices
 in $A$ and $B'$ denote the set of vertices that are out-neighbors of vertices in $B$. Observe that if $B =\emptyset$, then there are at most $|A|\le n\le 2n-3$ arcs. If $u\in B$, then its out-neighbors have only one incoming arc each. Therefore, there are at most $|B'|$ arcs going to $B'$, and at most $|A|$ arcs going to $A'$, altogether at most $|A|+|B'|$ arcs. If $A=\emptyset$, we are done; otherwise $A'$ is not empty, and since $A'\cap B'=\emptyset$, we have $|B'|\le n-1$, and since $B$ is not empty, we have $|A|\le n-1$. 

We are done unless $|A|=|B'|\ge n-1$. 
In that case $|B|=1$, thus $A=B'$. This means that there is an arc from $u\in B$ to each vertex in $A$, thus no arc goes between vertices of $A$. Therefore, each arc from $A$ goes to $u$, hence 
there are arcs $\overrightarrow{uv}$ and $\overrightarrow{vu}$, a contradiction. 
\end{proof}

\begin{proposition}
    Let $\overrightarrow{F}$ consist of an antidirected $P_3$ and an additional independent arc. Then $\exo(n,\overrightarrow{F})=2n-3$ if $n$ is sufficiently large. 
\end{proposition}

\begin{proof}
        Without loss of generality, assume that both arcs of the $P_3$ go to the middle vertex. For the lower bound,  consider an oriented graph with an arc $\overrightarrow{vu}$ such that $\overrightarrow{uw}$ and $\overrightarrow{vw}$ are arcs for each of the other $n-2$ vertices $w$. Then the copies of our antidirected $P_3$ consist of the arcs $\overrightarrow{uw}$, $\overrightarrow{vw}$ for some $w$, and there is no arc independent from them.

        For the upper bound, we are done unless there is a vertex with two incoming arcs. If these arcs are $\overrightarrow{uw}$, $\overrightarrow{vw}$, then other arcs contain at least one of $u,v,w$. If 
        $w$ is the only vertex of in-degree more than 1, then at most $n-1$ arcs go to vertices other than $w$. If at most 3 arcs go to $w$, then there are at most $n+2$ arcs altogether and we are done. 
        Otherwise each arc must be incident to $w$ and we are done. 


       Therefore, we can assume that there is another copy of our antidirected $P_3$. Let $U=\{u,v,w\}$ and $U'$ be the vertex set of another antidirected $P_3$.
       Then each arc intersects $U\cap U'$, except potentially the arcs inside these vertices. This means that for any vertex $x\not\in U\cup U'$, $x$ is only adjacent to vertices in $U\cap U'$. Clearly we find $\overrightarrow{F}$ if $|U\cap U'|=0$, and $|U\cap U'|\le 2$. We obtain the upper bound the upper bound $10+n-5$ if $|U\cap U'|=1$ and the upper bound $6+2(n-4)=2n-2$ if $|U\cap U'|=2$. We are done unless $|U\cap U'|=2$, let $x,y$ be the vertices in $U\cap U'$ and $x',y'$ be the two vertices in the symmetric difference of $U$ and $U'$. Then the arc between $x'$ and $y'$ is disjoint from the $2(n-4)$ arcs that are not inside $U\cup U'$ (since they are incident to $x$ or $y$). There is a vertex with two incoming arcs among those $2(n-4)$ arcs. These two arcs with the arc between $x'$ and $y'$ form $\overrightarrow{F}$, a contradiction, showing that the upper bound $2n-2$ is not sharp, thus there are at most $2n-3$ arcs.
\end{proof}

The above results, together with the fact that $z(\overrightarrow{T_3})=4$ 
complete the determination of all exact values of $\exo(n,\overrightarrow{F})$ for each oriented graph $\overrightarrow{F}$ with at most 3 arcs and sufficiently large $n$. 

\section{Stability}

Recall that Theorem \ref{stabi} states that an almost extremal graph is close to the Turán graph, but does not say anything about the orientation of the edges. It cannot say anything in general, since if $z(\overrightarrow{F})=\chi(F)$, then any orientation of $T(n,z(\overrightarrow{F})-1)$ is $\overrightarrow{F}$-free. Another problem is that multiple tournaments of order $z(\overrightarrow{F})-1$ may have the property that there is no homomorphism from $\overrightarrow{F}$ to them, which gives multiple $\overrightarrow{F}$-free orientations of $T(n,z(\overrightarrow{F})-1)$. Yet we can say something about the orientation of the Tur\'an graph in a special case.

\begin{theorem}\label{stab3}
    If $\overrightarrow{G}$ is an $\overrightarrow{P_3(t)}$-free oriented graph with at least $n^2/4-o(n^2)$ arcs, then we can turn $\overrightarrow{G}$ to the antidirected $K_{\lfloor n/2\rfloor,\lceil n/2\rceil}$ by deleting, adding and reorienting $o(n^2)$ arcs.
\end{theorem}

We note that we cannot extend this to every $\overrightarrow{F}$ with $z(\overrightarrow{F})=3$ and $\chi(F)=2$. Consider $K_{\lfloor n/2\rfloor,\lceil n/2\rceil}$ with the orientation such that one of the parts is partitioned to a set of vertices with in-degree 0 and a set of vertices with out-degree 0 (an unbalanced blow-up of $\overrightarrow{P_3}$). Then it does not contain $\overrightarrow{F}$ that is not a subgraph of any $\overrightarrow{P_3(t)}$, for example $\overrightarrow{F}$ with arcs $\overrightarrow{u_1u_2}$, $\overrightarrow{u_2u}$, $\overrightarrow{vu}$, 
$\overrightarrow{vv_1}$, $\overrightarrow{v_1v_2}$.

\begin{proof}
   By Theorem \ref{stabi}, we can turn the underlying graph $G$ to $K_{\lfloor n/2\rfloor,\lceil n/2\rceil}$ by deleting and adding  $o(n^2)$ edges, let $U$ and $W$ be the parts. We say that a set $U'\subset U$ is an \textit{out-set} if $|N_W^+(U')|\ge 2t$, and an \textit{in-set} if $|N_W^-(U')|\ge 2t$. We define these analogously for subsets of $W$. Observe that if $U'$ is both an out-set and an in-set and contains at least $t$ vertices, then we can pick $t$ common out-neighbors and $t$ common in-neighbors disjointly, which gives a $\overrightarrow{P_3(t)}$, a contradiction.

   Let us consider now the number of copies of $\overrightarrow{S_{2t,0}}$, where the $2t$ leaves are in $U$. On the one hand, this is equal to $\sum_{v\in W}\binom{d^-(v)}{2t}$. If there are quadratic many arcs going from $U$ to $W$, then this is $\Theta(n^{2t+1})$, while if subquadratic many arcs go from $U$ to $W$, then we are done. On the other hand, the number of such copies of $\overrightarrow{S_{2t,0}}$ is obtained by summing for $2t$-sets in $U$ the common out-neighborhood. Therefore, we are done unless there is a $2t$-set $U'$ with linear common out-neighborhood $W'$. We claim that $o(n^2)$ arcs go from $W'$ to $U$. Indeed, otherwise there is a $K_{2t,2t}$ using such edges. The part of that $K_{2t,2t}$ inside $W'$ is an in-set and an out-set, a contradiction.

   We have obtained that all but $o(n^2)$ pairs $u\in U$, $w\in W'$ have that $\overrightarrow{uw}$ is an arc in $G$. Let $U''$ denote the set of vertices in $U$ that have all but $o(n)$ vertices of $W'$ are outneighbors, then $|U''|=|U|-o(n)$. Observe that every $2t$-subset of $U''$ is an out-set. This implies that there are $o(n^{2t})$ $2t$-subsets of $U$ that are not out-sets, in particular there are $o(n^{2t})$ in-sets in $U$.

   Let us consider now the copies of $\overrightarrow{S_{0,2t}}$, where the $2t$ leaves are in $U$. On the one hand, there are $\sum_{v\in W}\binom{d^+(v)}{2t}$ copies. If there are quadratic many arcs going from $W$ to $U$, then this is $\Theta(n^{2t+1})$, while if subquadratic many arcs go from 
$W$ to $U$, then we are done.  On the other hand, the number of such copies of $\overrightarrow{S_{0,2t}}$ is obtained by summing for $2t$-sets in $U$ the common in-neighborhood. There are $o(n^{2t})$ $2t$-sets in $U$ that are in-sets, they have $O(n)$ common neighbors, and $O(n^{2t})$
 $2t$-sets in $U$ that are not in-sets, they have $O(1)$ common neighbors. Summarizing, the number of copies of such stars is $o(n^{2t+1})$, while we have obtained that it is $\Theta(n^{2t+1})$, a contradiction completing the proof.

\end{proof}

\begin{theorem}\label{thm3.2}
    Let $\overrightarrow{F}$ consist of the arcs $\overrightarrow{xy_i}$, $\overrightarrow{y_iz}$ for $i\le 2$. Then for sufficiently large $n$, we have $\exo(n,\overrightarrow{F})=\lfloor n^2/4\rfloor+\lceil n/2\rceil$.
\end{theorem}


\begin{proof}
    The lower bound is given by adding a directed cycle of length $\lceil n/2\rceil$ to the larger part of an antidirected $K_{\lfloor n/2\rfloor,\lceil n/2\rceil}$.

For the upper bound, assume that  $\exo(n,\overrightarrow{F})>\lfloor n^2/4\rfloor+\lceil n/2\rceil$ and consider an $\overrightarrow{F}$-free $n$-vertex oriented graph $\overrightarrow{G}$ with $\exo(n,\overrightarrow{F})$ arcs. First we show that we can assume that every vertex has total degree at least $n/2$. Otherwise, we delete a vertex of total degree less than $n/2$. If the resulting $(n-1)$-vertex graph has a vertex of total degree less than $(n-1)/2$, then we delete such a vertex, and so on. This process stops at a graph $\overrightarrow{G_1}$ on $n_1$ vertices. If $n_1$ is sufficiently large, then we apply the proof below to show that $\overrightarrow{G_1}$ has at most $\lfloor n_1^2/4\rfloor+\lceil n_1/2\rceil$ arcs. Therefore, $\overrightarrow{G}$ has at most $\lfloor n_1^2/4\rfloor+\lceil n_1/2\rceil+\lfloor (n_1+1)/2\rfloor+\lfloor (n_1+2)/2\rfloor+\dots+\lfloor n/2\rfloor\le \lfloor n^2/4\rfloor+\lceil n/2\rceil$ arcs. If $n_1$ is not sufficiently large, then $n_1=O(1)$ and $\overrightarrow{G}$ has at most $\binom{n_1}{2}+\lfloor (n_1+1)/2\rfloor+\lfloor (n_1+2)/2\rfloor+\dots+\lfloor n/2\rfloor\le \lfloor n^2/4\rfloor+\lceil n/2\rceil$ arcs.

This shows that we can assume each vertex has total degree at least $n/2$. Let us apply Theorem \ref{stab3}. We obtain that we can partition the vertex set of $\overrightarrow{G}$ to $U$ and $W$ such that for all but $o(n^2)$ pairs $u\in U$, $w\in W$ we have that $\overrightarrow{uw}$ is an arc of $\overrightarrow{G}$. We consider the partition with the most arcs going from $U$ to $W$; this implies that $u\in U$ has at least as many arcs going to $W$ as arcs coming from $U$, and analogously $w\in W$ has at least as many arcs coming from $U$ as arcs going to $W$.
Observe that $|U|,|W|=n/2+o(n)$.

Let $U'$ denote the set of vertices in $U$ that have outgoing arcs to $|W|-o(n)$ vertices of $W$.  Observe that for any two vertices in $U'$, there is a vertex in $W$ that is the out-neighbor of both vertices. Therefore, any vertex has at most one out-neighbor in $U'$ (otherwise $v$, its two outneighbors $u,u'$ in $U'$, and the common out-neighbor of $u$ and $u'$ create a forbidden configuration).  We define $W'$ analogously: those vertices in $W$ that have incoming arc from $|U|-o(n)$ vertices of $U$. Let $Q$ denote the set of vertices not in $U'\cup W'$, that is, $Q$ is those vertices where $\Omega(n)$ arcs are missing. Note that $|Q|=o(n)$, since the number of arcs from $U$ to $W$ that are not in $\overrightarrow{G}$ is $\Omega(|Q|n)$, but also $o(n^2)$ by the argument in the last paragraph.


Let $Q_1$ denote the set of vertices in $Q$ that have more than $|W|/2+1$ arcs going to $W$ and $Q_2$ denote the set of vertices in $Q$ that have more than $|U|/2+1$ arcs coming from $U$. Clearly, there is at most one vertex belonging to both $Q_1$ and $Q_2$. Let $Q_1'$ denote the set of other vertices of $Q_1$ and $Q_2'$ denote the set of other vertices of $Q_2$, then $Q_1'\subset U$ and $Q_2'\subset W$.

Observe that every vertex has at most one arc going to $U'\cup Q_1'$ and at most one arc coming from $W'\cup Q_2'$. In particular, if $v\in Q_3:=Q\setminus (Q_1\cup Q_2)$, then $v$ has at most one arc going to $U'\cup Q_1'$ and at most $|W|/2=n/4+o(n)$ arcs going to $W$, and similarly at most one arc coming from $W'\cup Q_2'$ and at most $|U|/2=n/4+o(n)$ arcs coming from $U$. Observe that we have listed all the possible arcs incident to $v$, except for the arcs going to $U\setminus (U'\cup Q_1')$ or coming from $W\setminus (W'\cup Q_2')$. Since there are $o(n)$ other vertices in $U\setminus (U'\cup Q_1')$ and in $W\setminus (W'\cup Q_2')$, we have that the total degree of $v$ is at most the number of arcs going to $W$ plus $n/4+o(n)$. Since the total degree of $v$ is at least $n/2$, $v$ has at least $n/4-o(n)$ arcs going to $W$, thus also to $W'\cup Q_2'$. We have already established the matching upper bound, thus $v$ has $n/4-o(n)$ arcs going to $W'\cup Q_2'$.  Analogously, $v$ has $n/4-o(n)$ arcs coming from $U'\cup Q_1'$.


We claim that $|Q_3|\le 4$. Indeed, 5 vertices of $Q_3$ have at least $5n/4-o(n)$ arcs coming from $U'\cup Q_1'$, thus there is a vertex of $U'\cup Q_1'$ with three out-neighbors in $Q_3$. Those 3 vertices have at least $3n/4$ arcs going to $W'\cup Q_2'$, thus two of them have a common neighbor in $W'\cup Q_2'$, a contradiction. Let $Q_3'$ denote the set we obtain by adding the at most one vertex of $Q_1\cup Q_2\setminus (Q_1'\cup Q_2')$ to $Q_3$, then $|Q_3'|\le 5$.

The fact that $|Q_3|\le 4$ further improves our bounds on the degrees of $v\in Q_3$: $v$ has more than $|U|/2-7$ arcs coming from $U'\cup Q_1'$ (and analogously more than $|W|/2-7$ arcs going to $W'\cup Q_2'$). Indeed, at most 1 arc comes from $W'\cup Q_2'$, at most one arc goes to $U'\cup Q_1'$, at most 4 arcs are inside $Q_3'$, and less than $|W|/2+1$ arcs go to $W'\cup Q_2'$, while $v$ has degree at least $n/2=(|U|+|W|)/2$.

Let us now partition $Q_i'$, $i\le 2$ into two parts. Let $P_1$ denote the vertices in $Q_1'$ with at least $|W|/2+7$ out-neighbors in $W$, and $R_1$ denotes the rest of the vertices in $Q_1'$. Observe that a vertex $w\in R_1$ has at most $|U|/2+1$ 
arcs coming from $U$ by the definition of $Q_1'$, at most $|W|/2+6$ arcs going to $W$, at most one arc coming from $W'\cup Q_2'$, and at most one arc going to $U'\cup Q_1'$. The remaining arcs incident to $w$ are also incident to a vertex of $Q_3'$, thus there are at most 5 such arcs. 
Therefore, $w$ has total degree at most $n/2+14$. 
Analogously, $P_2$ denotes the vertices in $Q_2'$ with at least $|U|/2+7$ in-neighbors in $U$, and $R_2$ denotes the rest of the vertices in $Q_2'$, then the vertices of $R_2$ have total degree at most $n/2+14$. 

Let $U_0$ denote the set of vertices in $U'\cup P_1$ that have an out-neighbor in $Q_3'$ and analogously, $W_0$ denote the set of vertices in $W'\cup P_2$ that have an in-neighbor in $Q_3'$. Observe that there is no arc inside $U_0$ or $W_0$, since that would mean that the endpoint of that arc and a vertex of $Q_3'$ have a common in-neighbor, while they clearly have a common out-neighbor. Furthermore, from a vertex of $U_0$, there is no arc to $U'\cup P_1$ by the same reasoning. We claim that there are at most $|(U'\cup P_1)\setminus U_0|$ arcs inside $U'\cup P_1$. Indeed, from each vertex there is at most one arc to vertices of $U'\cup P_1$, but from the vertices of $U_0$, there are no such arcs. Analogously, there are at most $|(W'\cup P_2)\setminus W_0|$ arcs inside $W'\cup P_2$.

CASE 1. Assume that there are at least two vertices $v,v'\in Q_3'$. They cannot have both common in-neighbors and out-neighbors, without loss of generality, the set of in-neighbors $A$ of $v$ and $A'$ of $v'$ are disjoint. Recall that both $v$ and $v'$ have at least $|U|/2-7$ in-neighbors inside $U$, thus all but at most 14 vertices of $U$ have an outgoing arc to $v$ or $v'$. Therefore, all but at most 14 vertices of $U'\cup P_1$ have an outgoing arc to $v$ or $v'$. 
Therefore, there are at most 14 arcs inside $U'\cup P_1$. Recall that every vertex of $Q_3$ has at least $|W|/2-7$ arcs going to $W$, thus $|W_0|\ge |W|/2-7$. 
Therefore, there are at most $|W|/2+7$ vertices inside $W'\cup P_2$ not in $W_0$, hence at most $|W|/2+7$ arcs inside $W'\cup P_2$. This implies that the total number of arcs inside $U'\cup P_1\cup W'\cup P_2$ is at most $|U'\cup P_1|| W'\cup P_2|+14+|(W'\cup P_2)\setminus W_0|\le |U'\cup P_1|| W'\cup P_2|+ |W|/2+21$.


Now we delete $Q_3$, and then some vertices of $R_1\cup R_2$ one by one. After the $i$th step, the remaining vertices form a graph $\overrightarrow{G}^{i}$ with vertex set formed by $U^i\subset U$ and $W^{i}\subset W$. If $R_1\cup R_2$ has a vertex with less than $|U^{i}|/2+7$ in-neighbors in $U^{i}$ and less than $|W^{i}|/2+7$ out-neighbors in $W^{i}$, then we delete such a vertex to obtain $\overrightarrow{G}^{i+1}$. Let $\overrightarrow{G}^j$ be the graph where we stop this. Then the remaining vertices are $U^j$, $W^j$ and at most one additional vertex, in $Q_3'\setminus Q_3$. 

CASE 1.1. There is no vertex in $Q_3'\setminus Q_3$. Then we repeat the argument of the first paragraph in CASE 1,  for $\overrightarrow{G}^j$. The set corresponding to $Q_3\cup R_1\cup R_2$ in $\overrightarrow{G}^j$ is empty, the set corresponding to $U'\cup P_1$ is $U^j$, and the set corresponding to $W'\cup P_2$ is $W^j$. Therefore, the above argument gives that there are at most $|U^j||W^j|+|(W'\cup P_2)\setminus W_0|+14$ arcs in $\overrightarrow{G}^j$.
We have deleted $x=o(n)$ vertices. Each time we deleted a vertex, its degree was at most half the number of vertices $n_i$ in the remaining graph plus 14. 
Now we build a new oriented graph $\overrightarrow{G^*}$. We add a vertex to $U^j$ with outgoing arc to each vertex of $W^j$, or add a vertex to $W^j$ with incoming arc from each vertex of $U^j$, and repeat this $x$ times. Each time we add the new vertex to the smaller part. Therefore, each time we added at least $n_i/2$ vertices. The resulting graph $\overrightarrow{G^*}$ consists of a bipartite graph on $n$ vertices and at most $|(W'\cup P_2)\setminus W_0|+14$ arcs inside the parts, thus contains at most $n^2/4+|(W'\cup P_2)\setminus W_0|+14$ arcs. Let us compare the number of arcs in $\overrightarrow{G}$ and $\overrightarrow{G^*}$. If we start from $\overrightarrow{G}^j$ and return the vertices of $\overrightarrow{G}$ one by one, each time we add back at most $n_i/2+14$ arcs, at most 14 more than when we build $\overrightarrow{G^*}$. This happens $x$ times, thus the total number of arcs in $\overrightarrow{G}$ is at most the total number of arcs in $\overrightarrow{G^*}$ plus $14x$. Therefore, the total number of arcs is at most $n^2/4+|(W'\cup P_2)\setminus W_0|+14\le n^2/4+|W|/2+21+14x=n^2/4+n/4+o(n)$ and we are done.

CASE 1.2. There is an additional vertex $z\in Q_3'\setminus Q_3$. Then we apply the same calculation as in CASE 1.1 to the graph obtained by deleting $z$. This gives that the total number of arcs in $\overrightarrow{G}$ is at most $(n-1)^2/4+14+|(W'\cup P_2)\setminus W_0|+14x+d$, where $d$ is the total degree of $z$.
Recall that our goal is to prove an upper bound of the form $(n-1)^2/4+n+O(1)$, thus we are done unless $d\ge n-|(W'\cup P_2)\setminus W_0|-o(n)$. 
Recall that every vertex has at most one arc going to $U'\cup Q_1'$ and at most one arc coming from $W'\cup Q_2'$. By the fact that $(W'\cup P_2)\setminus W_0 \subset W'\cup Q_2'$ and the definition of $W_0$, there is at most one arc from $(W'\cup P_2)\setminus W_0$ to $z$ and there is no arc from $z$ to $(W'\cup P_2)\setminus W_0$, 
so $z$ is adjacent to all but $o(n)$ of the other vertices of $\overrightarrow{G}$. By the facts that $U'\cup P_1\subset U'\cup Q'_1$ and $W_0\subset W'\cup Q'_2$, we have that at most one arc goes from $z$ to $U'\cup P_1$, and at most one arc goes from $W_0$ to $z$, thus, by our lower bound on $d$,
there are arcs from all but $o(n)$ vertices of $U'\cup P_1$ to $z$, and from $z$ to all but $o(n)$ vertices of $W_0$.  
Let $v\in Q_3=Q'_3\setminus \{z\}$. We have observed that $v$ has $n/4-o(n)$ arcs going to $W'\cup Q_2'$, and $v$ has $n/4-o(n)$ arcs coming from $U'\cup Q_1'$. 
By the definitions of $U_0$ and $W_0$, we conclude that all but $o(n)$ vertices in $U_0$ are in-neighbors of $z$, and roughly $n/4$ of them are in-neighbors of $v$. Similarly, all but $o(n)$ vertices in $W_0$ are out-neighbors of $z$, and roughly $n/4$ of them are out-neighbors of $v$. Hence, $v$ and $z$ have both common in-neighbors and out-neighbors, a contradiction.


CASE 2.
Assume that $|Q_3'|\le 1$. If there exists $v\in Q_3'$, then $v$ has at most one arc coming from $W'\cup Q_2$ 
and at most one arc going to $U'\cup Q_1$. 
The other at least $n/2-2$ arcs incident to $v$ go to $W$ or come from $U$, thus either at least $|U|/2-1$ arcs go from $U$ to $v$, or at least $|W|/2-1$ arcs go from $v$ to $W$. Without loss of generality, at least $|W|/2-1$ arcs go from $v$ to $W$ (in the case both at least $|U|/2-1$ arcs go from $U$ to $v$ and at least $|W|/2-1$ arcs go from $v$ to $W$, then without loss of generality, more arcs go from $v$ to $W$ than from $U$ to $v$). If $v\in W$, then we move $v$ to $U$; if $v\in U$, then we do not change these sets. More precisely, let $U^*:=U\cup \{v\}$ and $W^*:=W\setminus \{v\}$. If $Q_3'=\emptyset$, then $U^*=U$ and $W^*=W$.
Now, we have that any two vertices of $U^*$ have a common out-neighbor in $W^*$, and any two vertices of $W^*$ have a common in-neighbor in $U^*$. This implies that from each vertex of $U^*$ at most one arc goes to $U^*$, and to each vertex of $W^*$, at most one arc comes from $W^*$. 

There are at most $|U^*||W^*|$ arcs between $U^*$ and $W^*$, at most $|U^*|$ arcs inside $U^*$ and at most $|W^*|$ arcs inside $W^*$. Observe that $|U^*||W^*|+|U^*|\le \lfloor n^2/4\rfloor+\lceil n/2\rceil$ and $|U^*||W^*|+|W^*|\le \lfloor n^2/4\rfloor+\lceil n/2\rceil$. Therefore, it is enough to show that the number of arcs is at most $\max\{|U^*||W^*|+|U^*|,|U^*||W^*|+|W^*|\}$.
To that end, we will compare the number of arcs inside $U^*$ or $W^*$ to the number of pairs $u\in U^*$, $w\in W^*$ such that there is no arc between them. We say that such a pair is a \textit{good pair}. Let $a$ denote the number of arcs inside $U^*$ and $b$ denote the number of arcs inside $W^*$, then $a\le |U^*|$ and $b\le |W^*|$.
It is enough to show that the number of good pairs is at least $\min\{a,b\}$.

 We also say that $\overrightarrow{uw}$ with $u\in U^*$, $w\in W^*$ is a \textit{missing arc} if it is not in $\overrightarrow{G}$. Note that the endpoints of a missing arc do not necessarily form a good pair, since that arc can be in $\overrightarrow{G}$ backwards, i.e., the arc $\overrightarrow{wu}$ might be in $\overrightarrow{G}$.

Consider arcs $\overrightarrow{uu'}$ in $U^*$ and $\overrightarrow{ww'}$ in $W^*$. At least one of the arcs $\overrightarrow{uw}$ and $\overrightarrow{u'w'}$ is missing from $\overrightarrow{G}$. 
Let us consider the vertices in $U^*$ that are endpoints of arcs inside $U^*$. Let $u_1,\dots, u_s$ denote these vertices with $u_i$ having $a_i>0$ incoming arcs inside $U^*$. Analogously, $w_1,\dots, w_t$ are the vertices in $W^*$ such that $w_i$ has $b_i>0$ outgoing arcs. Consider now $u_i$ and $w_j$. Observe that either for all in-neighbors $u'$ of $u_i$ the arcs $\overrightarrow{u'w_j}$ are missing, or for all out-neighbors $w'$ of $w_j$ the arcs $\overrightarrow{u_iw'}$ are missing, i.e., at least $\min\{a_i,b_j\}$ arcs are missing between these $a_i+b_j+2$ vertices. 
Without loss of generality, the arcs $\overrightarrow{u'w_j}$ are missing. We claim that each missing arc is counted at most twice. Indeed, such an arc can be missing because of an inside arc in $U^*$ and an inside arc in $W^*$, where either both $u'$ and $w_j$ are starting points of the inside arcs, or both $u'$ and $w_j$ are endpoints of the inside arcs. If they are both starting points, then one of the arcs is $\overrightarrow{u'u_i}$, while the other arc starts with $w_j$. This means we considered the pair $u_i,w_j$ again. If they are both endpoints, then there is a single vertex $w_0$ such that $\overrightarrow{w_0w_j}$ is an arc, thus we considered the pair $u',w_0$.

Observe that $a=\sum_{i=1}^s a_i$ and $b=\sum_{j=1}^t b_j$. 
Assume without loss of generality that $\max\{a_i,b_j\}=a_1$.
Consider the pairs $u_1,w_j$ for each $j\le t$. For each $j$, we find at least $b_j$ missing arcs of the form $\overrightarrow{u'w_j}$ or $\overrightarrow{u_1w'}$. 
More precisely, in the case of the arcs of the form $\overrightarrow{u'w_j}$, we found $a_1\ge b_j$ missing arcs, and at most one of them goes backwards. In the case the arcs of the form $\overrightarrow{u_1w'}$ are missing, we found $b_j$ missing arcs, and for all such $j$, at most one arc goes backwards altogether, since any two vertices of $W$ have a common in-neighbor, thus cannot have a common out-neighbor. 
This way we found at least $b_j-1$ pairs with no arcs in any direction between them. Moreover, we found at least $b_j$ such pairs if we found the missing arcs of the form $\overrightarrow{u'w_j}$ and $b_j<a_1$ or if there is no arc that goes backwards. Thus, not counting the at most one arc that goes backwards to $u_1$, the only case we found less than $b_j$ good pairs is if $b_j=a_1$ and the missing arcs contain $w_j$.

Let $t'$ denote the number of vertices $w_j$ with $b_j=a_1$.
Repeating the above argument for each of these vertices, we obtain at least $b-t'-1$ good pairs. Observe that so far, we have counted each good pair only once, since when we count a good pair twice, then we count it for pairs $u_i,w_j$ and $u_{i'},w_{j'}$ with $i'\neq i$, $j'\neq j$.  
Let $s'$ denote the number of vertices $u_i$ with $a_i=a_1$.

CASE 2.1. Assume that $s'\ge 2$. 

Case 2.1.1. 
If $a_1\le 2$, then we have $a\le 2s$ and $b\le 2t$, without loss of generality $s\ge t$. We find at least $st/2$ missing arcs and at most $t$ backward arcs. If $s<6$, then the number of arcs is at most $|U^*||W^*|+a+b\le n^2/4+20$ and we are done. Otherwise, there are at least $(s-2)t/2\ge 2t\ge b$ good pairs, and we are done.

Case 2.1.2. Let $a_1\ge 3$.
Without loss of generality $a_2=a_1$, and repeat the above with $u_2$. 
Without loss of generality, $\overrightarrow{u_2u_1}$ is not an arc in $\overrightarrow{G}$ ($\overrightarrow{u_1u_2}$ may be an arc or a non-arc). Consider now $j$ with $b_j=a_1$. Our first goal is to show that for $u_1,w_j$ and $u_2,w_j$ we find at least $b_j$ good pairs (that is, we find at least $b_j$ good pairs totally when we consider the pairs $u_1,w_j$ and $u_2,w_j$) not counted for other $w_i$, 
unless the arc that goes backwards to $u_1$ starts at an in-neighbor of $w_j$. This will improve our lower bound on the number of good pairs from $b-t'-1$ to $b-1$.

If the pair $u_1,w_j$ already gave us at least $b_j$ good pairs in the above argument, then there is nothing to prove. Assume now that
the pair $u_1,w_j$ gave us fewer than $b_j$ (i.e., exactly $b_j-1$) good pairs in the above argument, and not because of the arc going backwards to $u_1$. Then we found $b_j$ missing arcs of the form $\overrightarrow{u'w_j}$, where $u'$ is an in-neighbor of $u_1$. Consider now the pair $u_2,w_j$. By the above argument, we find at least $b_j$ missing arcs, where the vertex in $U$ is either $u_2$, or an in-neighbor of $u_2$
. Recall that $u_2$ or an in-neighbor of $u_2$ cannot be an in-neighbor of $u_1$. Therefore, the set of missing arcs we found for $u_1,w_j$ is disjoint from the set of missing arcs we found for $u_2,w_j$. In other words, we find at least $b_j\ge 3$ missing arcs for the pair $u_2,w_j$ that was not found for the pair $u_1,w_j$. We achieved our first goal 
unless all of these $b_j$ missing arcs were counted for some other $w_i$ or go backwards.  At most one of them goes backwards, thus at least $b_j-1\ge 2$ was counted twice. Since they were counted for $u_2,w_j$, the only way to count them another time is if we count them for $u_1,w_i$. If we count a missing arc twice, then we count it for some $u\in U^*$ and an in-neighbor of $u$ in $U^*$, 
thus $u_1$ is an in-neighbor of $u_2$
, $w_i$ is an in-neighbor of $w_j$ and the missing arc is $\overrightarrow{u_1w_j}$. However, this contradicts our assumption that  we found $b_j$ missing arcs of the form $\overrightarrow{u'w_j}$, where $u'$ is an in-neighbor of $u_1$.

We have found $b-1$ good pairs, we only need to find one more. Moreover, if there is no backward arc to $u_1$ (counted for some pair $u_1,w_i$
), then we have found $b$ good pairs. Without loss of generality, the backward arc is $\overrightarrow{w'u_1}$ for some out-neighbor $w'$ of $w_i$, and is counted for the pair $u_1,w_i$. Recall that for each $j\neq i$, for the pairs $u_1,w_j$ and $u_2,w_j$ we have found at least $b_j$ good pairs (not counted for other pairs), thus we are done if the same also holds for $w_i$. 

We have found at least $b_i-1$ good pairs for the pair $u_1,w_i$, of the form $u_1,w$ for some $w\in W^*$ that is an out-neighbor of $w_i$. For the pair $u_2,w_i$, we have found missing arcs either of the form $\overrightarrow{u'w_i}$ for some in-neighbor $u'\in U^*$ of $u_2$, or of the form $\overrightarrow{u_2w^*}$ for some $w^*\in W^*$ that is an out-neighbor of $w_i$. In the first case, we have found $a_1$ missing arcs, but one of them may have been counted, since $u'$ may be $u_1$. Another of these missing arcs may go backwards. Therefore, we found $a_1-2$ new good pairs and we are done. In the second case, these missing arcs have not been counted for pairs $u_1,w_i$, thus we found $b_i$ new missing arcs, and at least $b_i-1$ new good pairs. We are done, unless $b_i=1$ and there is an arc going backwards to $u_2$ from the only out-neighbor $w'$ of $w_i$. But there is an arc from $w'$ to $u_1$, a contradiction.

CASE 2.2. We have $s'\le 1$. Then $s'= 1$, as $s'\geq 1$ by our assumption that $\max\{a_i,b_j\}=a_1$.  Observe that if $t'\ge 2$, then we can repeat CASE 2.1 with $W$ replacing $U$ and complete the proof, thus we can assume that $t'\le 1$. Therefore, we have found at least $b-2$ good pairs when considering the pairs $u_1,w_j$. Our first goal is to improve this bound to $b-1$. If $t'=0$, this is clear. Now assume $t'=1$. Recall that since the forbidden graph is symmetric, originally the role of $U^*$ and $W^*$ is symmetric. This symmetry was broken by the assumption that $\max\{a_i,b_j\}=a_1$, i.e., by the assumption that the largest value of $a_i$ and $b_j$ inside parts belongs to a vertex in $U^*$. In the current case where $s'=1$ and $t'=1$, the symmetry is restored. 
Let $b_1=a_1$, and consider the pair $u_1,w_1$. The corresponding missing arcs each contain either $u_1$ or $w_1$. Because of the symmetry, we can assume that these missing arcs contain $u_1$. Recall that for a pair $u_1,w_j$, we found at least $b_j$ good pairs, unless one of the missing arcs went backwards to $u_1$ or $b_j=a_1$, the missing arcs contain $w_j$ and one of those missing arcs goes backwards. It is clear now that this second possibility does not occur, thus we found at least $b-1$ good pairs. 
Moreover, for each $j$ we found either at least $b_j$ good pairs or at least $b_j-1$ good pairs and an arc that goes backwards to $u_1$. Therefore, we are done unless there is an arc backwards to $u_1$.




If 
there are at most $|U^*|-1$ arcs inside $U^*$, then the total number of arcs is at most $|U^*||W^*|-(b-1)+|U^*|-1+b$ and we are done. Therefore, we can assume that there are $|U^*|$ arcs inside $U^*$. Since every vertex of $U^*$ has at most one out-neighbor in $U^*$, it is only possible if every vertex of $U^*$ has exactly one out-neighbor in $U^*$. In particular, this implies that there is a directed cycle inside $U^*$.

Assume first that there is a directed cycle inside $U^*$ that contains $u_1$. Then there are arcs $\overrightarrow{u'u''}$
and $\overrightarrow{u''u_1}$. Recall that there is an arc $\overrightarrow{wu_1}$ for some $w\in W^*$. Then $\overrightarrow{u'w}$ is missing, otherwise there are two two-arc paths from $u'$ to $u_1$. But it cannot go backwards, since there is at most one backward arc incident to $w$. Therefore, $u',w$ is a good pair. We have counted so far only good pairs incident to $u_1$ or an in-neighbor of $u_1$. Observe that $u'$ is not $u_1$, nor the only out-neighbor of $u'$ is $u_1$. Therefore, $u',w$ is a new good pair, we found at least $b$ good pairs, completing the proof.

Assume now that there is a directed cycle inside $U^*$ that does not contain $u_1$. Then it cannot contain in-neighbors of $u_1$ either, since the only one out-neighbor of each vertex of the cycle must also be in this cycle. An arbitrary vertex in this cycle, paired with any $w_j$ gives a missing arc. That missing arc has not been counted, since all the missing arcs we counted contain $u_1$ or an in-neighbor of $u_1$. We are done, unless that missing arc is backwards, say $\overrightarrow{wu}$. Similar to the previous paragraph, we have arcs $\overrightarrow{u'u''}$ and $\overrightarrow{u''u}$ in the cycle, and then $u',w$ is a good pair by the same reasoning. This is a new good pair, completing the proof.
\end{proof}



\section{Antidirected oriented graphs with bounded out-degrees}

In this section, we will prove an oriented version of F\"{u}redi-Alon-Krivelevich-Sudakov Theorem.

\begin{definition}\label{def2}
Let $G$ be a bipartite digraph with bipartition $U\cup W$. A subset $R\subseteq W$ is {\em $(r,h)$-rich} if each $R'\subseteq R$ with $|R'|=r$ has at least $h$ common in-neighbors in $U$.
\end{definition}

We use $G=G(U\cup W, U \rightarrow W)$ to denote a bipartite digraph $G$ with bipartition $U\cup W$ such that all arcs of $G$ are from $U$ to $W$. For each $u\in U$, we use $N_W^+(u)$ to denote the set of out-neighbours of $u$ in $W$. Similarly, for each $w\in W$, we use $N_U^-(w)$ to denote the set of in-neighbours of $w$ in $U$.

\begin{lemma}\label{proRICHSET}
Let $H=H(A\cup B, A \rightarrow B)$ be an $h$-vertex bipartite digraph such that $|N_B^+(a)|\leq r$ for each $a\in A$. Let $G=G(U\cup W, U \rightarrow W)$ be a bipartite digraph. If $W$ contains an $(r,h)$-rich set of size at least $h$, then $G$ contains a copy of $H$.
\end{lemma}
\begin{proof} It suffices to embed $H$ into $G$. Let $R$ be an $(r,h)$-rich set of $W$ of size at least $h$ and define map $\varphi:B\rightarrow R$ such that we can embed $B$ into $R$. By the definition of an $(r,h)$-rich set, each $r$-set $R'\subseteq R$ has at least $h$ common in-neighbors. This implies that for any vertex $a\in A$, we can map $a$ to an unused vertex in $N_G^-(\varphi(N_{H}^+(a)))$. Hence, $H$ can be embeded into $G$.
\end{proof}

\begin{lemma}\textbf{(Embedding Lemma)}\label{lemma~embedding}
Let $H=H(A\cup B, A \rightarrow B)$ be an $h$-vertex bipartite digraph such that $|N_B^+(a)|\leq r$ for each $a\in A$. Let $G=G(U\cup W, U \rightarrow W)$ be a bipartite digraph such that $|U|>h\dbinom{|W|}{ r}$ and $|N_W^+(u)|\geq h$ for each $u\in U$. Then $W$ contains an $(r,h)$-rich set of size at least $h$. In particular, $G$ contains a copy of $H$.
\end{lemma}
\begin{proof} 
Take a maximal partial coloring map $\varphi: U\rightarrow \dbinom{W}{ r}$ (all vertices in $U$ can be viewed as colors assigned to $r$-subsets of $W$) such that the following statements hold:

\begin{enumerate}
\item For any $u\in U$, if $\varphi(u)=R$, then $R\subseteq N_W^+(u)$, where $R$ is an $r$-subset of $W$.\\
\item For any $R\in \dbinom{W}{ r}$,~$\varphi^{-1}(R)$ has size at most $h$. That is, we do not assign more than $h$ colors to any $R$.\\
\item $\varphi$ is injective. That is, $|\varphi(u)|\leq 1$ for any $u\in U$.
\end{enumerate}
As $|U|>h\dbinom{|W|}{ r}$, there exists some vertex $u'\in U$ which is not assigned to any $r$-set in $W$.

\begin{claim}\label{cla0}
Each $h$-set in $N_W^+(u')$ is $(r,h)$-rich.
\end{claim}
\begin{proof}
Let $W'\subseteq N_W^+(u')$ with $|W'|=h$. It suffices to show that any $r$-set $T\subseteq W'$ has at least $h$ common in-neighbors in $U$. Suppose $T$ has less than $h$ common in-neighbors in $U$. Obviously, $u'$ can be viewed as a color assigned to $T$, that is, $\varphi(u')=T$ holds. It contradicts the maximum of $\varphi$. Hence, $T$ has at least $h$ in-coneighbors in $U$ and so $W'$ is $(r,h)$-rich.  This completes the proof of Claim~\ref{cla0}.
\end{proof}

This claim and Lemma~\ref{proRICHSET} imply that $G$ contains a copy of $H$. This completes the proof of Lemma~\ref{lemma~embedding}.
\end{proof}

The following is an oriented version of a random zooming theorem from \cite{Fernandez-Hyde-Liu-Pikhurko-Wu-2023}.

\begin{theorem}\textbf{(Oriented Random Zooming Theorem)}\label{zooming}
Let $d\geq max\{40,2h\}$. Let $G=G(U\cup W, U\rightarrow W)$ be a bipartite digraph such that $|N_W^+(u)|\geq d$ for each $u\in U$. Let $H=H(A\cup B,A\rightarrow B)$ be an $h$-vertex bipartite digraph such that $|N_B^+(a)|\leq r$ for each $a\in A$. If$$\frac{1}{2|W|}\left(\frac{|U|}{4h}\right)^{1/r}\frac{d}{2}\geq \max\{20,h\},$$
then $G$ contains a copy of $H$.
\end{theorem}
\begin{proof}
Let $p=\frac{1}{2|W|}\big(\frac{|U|}{4h}\big)^{1/r}$. If $|U|> 4h(2|W|)^r$, then we replace $U$ by a subset of size exactly $4h(2|W|)^r$, now the inequality $$\frac{1}{2|W|}\left(\frac{|U|}{4h}\right)^{1/r}\frac{d}{2}\geq \max\{20,h\}$$ still holds and we have $p=1$. Otherwise, we have $p<1$.

Let $W'\subseteq W$ be a $p$-random subset of $W$, that is, each vertex is chosen independently at random from $W$ with probability $p$. For each $u\in U$, set $$X_u:=|N_{W'}^+(u)|=\sum_{w\in N_W^+(u)}\textbf{1}_{w\in W'}.$$ Note that $X_u$ is a sum of independent Bernoulli random variables of parameter $p$ with expectation $$\mathbb {E}(X_u)=pd^+(u)\geq pd.$$ Then using a standard lower-tail Chernoff bound and $pd/12\geq 2$ gives that for each $u\in U$, we have
$$\mathbb{P}r(X_u<pd/2)\leq e^{-pd/12}\leq 1/4.$$
Let $$U'=\{u\in U\mid X_u\geq pd/2\}.$$ We will show that Claim~\ref{cla1} holds. We see that if Claim~\ref{cla1} holds, then there exists some choice of $W'$ such that both $|W'|\leq 2p|W|$ and $|U'|\geq |U|/4$ hold. Furthermore, combining with this and the definition of $p$, we have $$|U'|\geq \frac{|U|}{4}=h(2p|W|)^r\geq h(|W'|)^r>h\dbinom{|W'|}{ r}.$$
Therefore, $|U'|>h\dbinom{|W'|}{ r}$. By definitions of $U'$ and $p$, and the inequality $$\frac{1}{2|W|}\left(\frac{|U|}{4h}\right)^{1/r}\frac{d}{2}\geq \max\{20,h\},$$ for each $u'\in U'$, we have $$|N_W^+(u')|\geq |N_{W'}^+(u')|=X_u\geq pd/2=\frac{1}{2|W|}\left(\frac{|U|}{4h}\right)^{1/r}\frac{d}{2}\geq h.$$ By Lemma~\ref{lemma~embedding}, $G$ contains a copy of $H$. 
\end{proof}

It remains to prove the following claim.

\begin{claim}\label{cla1}
$\mathbb{P}r(|W'|>2p|W|)+\mathbb{P}r(|U'|<|U|/4)<1.$
\end{claim}
\begin{proof}
To prove this claim, we will show that $\mathbb{P}r(|U'|<|U|/4)\leq 1/2$ and $\mathbb{P}r(|W'|>2p|W|)\leq 1/4$. 

We first prove that $q:=\mathbb{P}r(|U'|<|U|/4)\leq 1/2$. Suppose that $q>1/2$. Note that $$\mathbb{E}(|U'|)=|U|\mathbb{P}r(X_u\geq pd/2)\geq \frac{3|U|}{4}.$$
On the other hand, $$\mathbb{E}(|U'|)\leq q|U|/4+(1-q)|U|\leq \frac{5|U|}{8},$$
a contradiction. Hence, $\mathbb{P}r(|U'|<|U|/4)\leq 1/2$.

We next prove that $\mathbb{P}r(|W'|>2p|W|)\leq 1/4$. By the facts that $|W|\geq |N_W^+(u)|\geq d$ and  $$pd/2=\frac{1}{2|W|}\left(\frac{|U|}{4h}\right)^{1/r}\frac{d}{2}\geq \max\{20,h\},$$ we have $$\mathbb{E}(|W'|)=p|W|\geq pd\geq 40.$$
Combining this with an upper-tail Chernoff bound implies $$\mathbb{P}r(|W'|>2p|W|)\leq e^{-p|W|/3}\leq e^{-\frac{40}{3}}\leq 1/4.$$ This completes the proof of Claim~\ref{cla1}.
\end{proof}

\begin{lemma}\label{subdigraph1}
For any oriented graph $G$, there exists a spanning bipartite subdigraph $H=H(X\cup Y, X\rightarrow Y)$  such that  $|E(H)|\geq \frac{|E(G)|}{4}$ and $\big||X|-|Y|\big|\leq 1$.
\end{lemma}
\begin{proof}
Without loss of generality, we assume that $|G|$ is even. Now randomly partition the vertex set $V(G)$ into two subsets $X$ and $Y$ with $|X|=|Y|=\frac{|V(G)|}{2}$. For each $e\in E(G)$, we use $\mathcal{A}_e$ to denote the event that $e$ is from  $X$ to $Y$ in $H$. Observe that $\mathbb{P}r(\mathcal{A}_e)=\frac{1}{4}$.
Let \[
X_e=\left\{
   \begin{array}{ll}
     1, &\mbox {if $\mathcal{A}_e$ occurs;}\\
     0, &\mbox {otherwise.}
   \end{array}
   \right.
\]
Clearly, $E(H)=\sum_{e\in E(G)}X_e$. By the linearity of expectation, we have
$$\mathbb{E}[E(H)]=\mathbb{E}[\sum_{e\in E(G)}X_e]=\sum_{e\in E(G)}\mathbb{E}[X_e]=\sum_{e\in E(G)}\mathbb{P}r(\mathcal{A}_e)=\frac{|E(G)|}{4}.$$
Hence, there is a choice of $X$ and $Y$, such that $H$ is a spanning bipartite subdigraph $H=H(X\cup Y, X\rightarrow Y)$  such that  $|E(H)|\geq \frac{|E(G)|}{4}$ and $\big||X|-|Y|\big|\leq 1$, as desired.
\end{proof}

For a digraph $D$, let $$\Delta(D)=\max\{d(v): v\in V(D)\},d(D)=2|E(D)|/n, \delta(D)=\min\{d(v): v\in V(D)\},$$  where $d(v)=d^+(v)+d^-(v)$. A digraph $D$ is called {\em $K$-almost regular} if $\Delta(D)\leq K\delta(D)$. We need to show the following result on $K$-almost regular subdigraph which could be seen as an oriented version of  $K$-almost regular subgraph theorem by Erd\H{o}s and Simonovits~\cite{Erdos-Simonovits1969}.

\begin{lemma}\label{almost regular}
Let $\varepsilon=1-\frac{1}{r}\in (0,1),~c>0,~K=20\cdot 2^{\frac{1}{{\varepsilon}^2}+1}$. Let $H=H(X\cup Y, X\rightarrow Y)$ be a bipartite digraph such that $|E(H)|= \frac{c}{4}n^{1+\varepsilon}$ and $\big||X|-|Y|\big|\leq 1$, where $n=|H|$ is sufficiently large. Then $H$ contains an $n_s$-vertex $K$-almost regular subdigraph $H_s$ with $n_s \geq \frac{8}{cK}n^{\frac{\varepsilon-
\varepsilon^2}{4+4\varepsilon}} $ and $|E(H_s)|\geq \frac{c}{10}{n_s}^{1+\varepsilon}$.
\end{lemma}
\begin{proof}
Let $t=K/20=2^{\frac{1}{\varepsilon^2}+1}$. Partition $V(H)$ into $2t$ subsets, say $V_1,V_2,\ldots, V_{2t},$ such that each of them has equal size (that is, $|V_i|=n/2t$ for any $i\in[2t]$), and $V_1$ contains the highest degree vertices. We divide the proof into the following two cases.\\

Case 1: At most half of the arcs are incident to $V_1$ in $H$. Let $H_0=H-V_1$. Clearly, $$e(H_0)\geq e(H)-\frac{e(H)}{2}= \frac{c}{8}n^{1+\varepsilon}.$$
We  repeatedly delete vertices of degree less than $d_0=\frac{c}{40}n^\varepsilon$ until no such vertices exist. We denote the final graph as $H_s$ and let $n_s=|V(H_s)|$. Now we have $$e(H_s)\geq e(H_0)-nd_0=\frac{c}{8}n^{1+\varepsilon}-\frac{c}{40}n^\varepsilon\cdot n=\frac{c}{10}n^{1+\varepsilon}\geq \frac{c}{10}{n_s}^{1+\varepsilon}.$$
By the facts that $$\delta(H_s)\geq d_0=\frac{c}{40}n^\varepsilon$$ and $$\Delta(H_s)\leq d(V_1) \leq \frac{\frac{c}{8}n^{1+\varepsilon}\cdot 2}{n/2t}=\frac{ct}{2}\cdot n^\varepsilon,$$ we have
$\Delta(H_s)\leq K\delta(H_s)$. Hence, $H_s$ is a $K$-almost regular digraph. \\
Since $$\frac{1}{5}n^{1+\varepsilon}\leq 2e(H_s)\leq n_s\Delta(H_s)\leq n_s\cdot \frac{ct}{2} n^\varepsilon,$$ we have $$n_s\geq \frac{2}{5ct}n=\frac{8}{cK}n> \frac{8}{cK}n^{\frac{\varepsilon-
\varepsilon^2}{4+4\varepsilon}}.$$


Case 2: At least half of the arcs are incident to $V_1$ in $H$. By the pigeonhole principle, there exists $V_i~(i\in \{2,3,\ldots,2t\})$ such that 
$$e(H[V_1,V_i])\geq \frac{\frac{c}{8}n^{1+\varepsilon}}{2t}=\frac{c}{16t}n^{1+\varepsilon}=\frac{e(H)}{4t}.$$
Let $m_1=|V(H[V_1,V_i])|$. We apply the above (splitting) method to $H[V_1,V_i]=H^{m_1}$ again. Either we obtain a $K$-almost regular digraph like in Case 1 or an $H^{m_2}$ like $H^{m_1}$ in Case 2. In the latter case, we use an argument similar to $H^{m_2}$ and so on, then we prove that this process will terminate at a large subdigraph satisfying the condition of Case 1.

Indeed, for the graph $H^{m_k}$, we have $$e(H^{m_k})\geq \frac{1}{(4t)^k}e(H)= \frac{1}{(4t)^k}\cdot \frac{c}{4}n^{1+\varepsilon},$$
and $$  m_k\thickapprox \frac{1}{t^k}n.$$
Therefore,
$$\frac{1}{(4t)^k}\cdot \frac{c}{4}n^{1+\varepsilon}\leq e(H^{m_k}) \leq {m_k}^2 \thickapprox (\frac{1}{t^k}n)^2=\frac{1}{t^{2k}}n^2.$$
Thus $$\left(\frac{t}{4}\right)^k\leq \frac{4}{c}n^{1-\varepsilon},$$
and consequently$$k\leq \frac{\lg ({\frac{4}{c}n^{1-\varepsilon}})}{\lg\left({\frac{t}{4}}\right)}.$$
This means that the splitting procedure will stop after at most $\frac{\lg \left({\frac{4}{c}n^{1-\varepsilon}}\right)}{\lg\left({\frac{t}{4}}\right)}$ steps.

On the other hand,
$$\lg{m_k}\thickapprox \lg{\frac{n}{(t)^k}}=\lg{n}-k\lg{t}.$$
Therefore,
\begin{equation*}
\begin{aligned}
\lg{m_k}&\geq \lg{n}- \frac{\lg ({\frac{4}{c}n^{1-\varepsilon}})}{\lg\left({\frac{t}{4}}\right)}\cdot\lg{t}\\
&=\lg{n}- \frac{\lg {\left(\frac{4}{c}\right)+(1-\varepsilon)\lg{n}}}{\lg\left({\frac{t}{4}}\right)}\cdot \lg{t}\\
&=\left(1-(1-\varepsilon)\frac{\lg{t}}{\lg({\frac{t}{4})}}\right)\lg{n}-\frac{\lg{\left(\frac{4}{c}\right)}\lg{t}}{\lg{\left(\frac{t}{4}\right)}}\\
&\geq \frac{1}{4}\left(1-(1-\varepsilon)\frac{\lg{t}}{\lg{\left(\frac{t}{4}\right)}}\right)\lg{n}\\
&=\frac{1}{4}\left(1-(1-\varepsilon)\frac{\frac{1}{\varepsilon^2}+1}{\frac{1}{\varepsilon^2}-1}\right)\lg{n}\\
&=\frac{\varepsilon-\varepsilon^2}{4+4\varepsilon}\lg{n}.\\
\end{aligned}
\end{equation*}
Note that the second inequality holds since $\frac{\lg{\left(\frac{4}{c}\right)}\lg{t}}{\lg{\left(\frac{t}{4}\right)}}$ is a constant.

Now let $\varepsilon'=\frac{\varepsilon-
\varepsilon^2}{4+4\varepsilon}$, we have
$$m_k\geq n^{\varepsilon'}=n^{\frac{\varepsilon-
\varepsilon^2}{4+4\varepsilon}}.$$
This means that the resulting digraph $H_s$ will have at least $\frac{8}{cK}n^{\varepsilon'}$ vertices. 
Clearly, $H_s$ is $K$-almost regular and $e(H_s)\geq \frac{c}{10}{n_s}^{1+\varepsilon}$.
\end{proof}

Now we are ready to prove Theorem \ref{main} which we restate here for convenience.

\begin{thm*}
Let $H=H(A\cup B,A\rightarrow B)$ be an $h$-vertex bipartite digraph such that $|N_B^+(a)|\leq r$ for each $a\in A$. Then there exists a constant $c=c(H)$ such that $$\exo(n,H)\le c\cdot n^{2-\frac{1}{r}}.$$
\end{thm*}
\begin{proof} We will pick $c$ later.
Let $G$ be an $n$-vertex oriented graph with more than $cn^{2-\frac{1}{r}}$ arcs. By Lemma~\ref{subdigraph1}, there exists a spanning bipartite subdigraph $G'=G'(X\cup Y,X\rightarrow Y)$ such that $$|E(G')|\geq \frac{|E(G)|}{4}\geq \frac{c}{4}n^{2-\frac{1}{r}}, \big||X|-|Y|\big|\leq 1.$$ 

We can slightly decrease $c$ such that $\frac{c}{4}n^{2-\frac{1}{r}}$ is an integer and we can assume that $|E(G')|= \frac{c}{4}n^{2-\frac{1}{r}}$ by deleting arcs if necessary.
Let $\varepsilon=1-\frac{1}{r}\in (0,1)$ and $K=20\cdot 2^{\frac{1}{{\varepsilon}^2}+1}$. By Lemma~\ref{almost regular}, $G'$ contains an $n_s$-vertex $K$-almost regular subdigraph $H_s$ with order $n_s=|H_s|> \frac{8}{cK}n^{\frac{\varepsilon-\varepsilon^2}{4+4\varepsilon}}$ and size $|E(H_s)|\geq \frac{c}{10}{n_s}^{1+\varepsilon}$. Moreover, $H_s=H_s(X_s\cup Y_s,X_s\rightarrow Y_s)$ is an oriented bipartite graph. 
Therefore, 
$$K\delta(H_s)\geq \Delta(H_s) \geq d(H_s)\geq \frac{c}{5}{n_s}^{\varepsilon}\geq \delta(H_s).$$
Hence, there exist two positive constants $K_1$ and $K_2$ such that $$K_1\delta(H_s)=d(H_s)=K_2\Delta(H_s)\geq\frac{c}{5}{n_s}^{\varepsilon}.$$
Therefore, the number of vertices in $X_s$ and $Y_s$ satisfy the following inequality
$$\frac{K_2|E(H_s)|}{d(H_s)}=\frac{|E(H_s)|}{\Delta(H_s)}\leq |X_s|,|Y_s|\leq \frac{|E(H_s)|}{\delta(H_s)}=\frac{K_1|E(H_s)|}{d(H_s)}.$$
Hence, we can get
$$\frac{K_2}{K_1}\leq \frac{|X_s|}{|Y_s|}\leq \frac{K_1}{K_2}.$$

Now we pick $c$ to be slightly larger than $$c(H)=\max\{20,h\}\cdot 20{K_1}^{1+\frac{1}{r}}\cdot \left(\frac{K_2}{4h}\right)^{\frac{1}{r}}\cdot \left(\frac{K_1}{K_2+K_1}\right)^{1-\frac{1}{r}},$$
so that for sufficiently large $n$, even after decreasing $c$ slightly (to ensure that $\frac{c}{4}n^{2-\frac{1}{r}}$ is an integer), we have that $c\ge c(H)$. We let $d=\delta(H_s).$

Now we have $$|N_{Y_s}^+(u)|\geq d\geq\frac{c}{5K_1}{n_s}^\varepsilon=\frac{c}{5K_1}{n_s}^{1-\frac{1}{r}} \geq \max\{40,2h\}$$ for each $u\in X_s$. Furthermore, the following holds

\begin{equation*}
\begin{aligned}
\frac{1}{2|Y_s|}\cdot \left(\frac{|X_s|}{4h}\right)^{\frac{1}{r}}\cdot \frac{d}{2}&\geq \frac{1}{2|Y_s|}\cdot \left(\frac{|X_s|}{4h}\right)^{\frac{1}{r}}\cdot\frac{\frac{c}{5K_1}{n_s}^{1-\frac{1}{r}}}{2}\\
&=\frac{c}{20K_1}\cdot \frac{1}{(4h)^{\frac{1}{r}}}\cdot \left(\frac{|X_s|}{|Y_s|}\right)^{\frac{1}{r}}\cdot \left(\frac{n_s}{|Y_s|}\right)^{1-\frac{1}{r}}\\
&\geq \frac{c}{20K_1}\cdot \frac{1}{(4h)^{\frac{1}{r}}}\cdot \left(\frac{K_2}{K_1}\right)^{\frac{1}{r}}\cdot \left(1+\frac{K_2}{K_1}\right)^{1-\frac{1}{r}}\\
&=\frac{c}{20K_1^{1+\frac{1}{r}}}\cdot \left(\frac{K_2}{4h}\right)^{\frac{1}{r}}\cdot \left(\frac{K_2+K_1}{K_1}\right)^{1-\frac{1}{r}}\\
&\geq \max\{20,h\}.\\
\end{aligned}
\end{equation*}

By Theorem~\ref{zooming}, $H_s$ contains a copy of $H$, and so $G$ contains a copy of $H$. Hence, $\exo(n,H)\leq cn^{2-\frac{1}{r}}$. This completes the proof.
\end{proof}
Note that this bound is sharp, as shown by the antidirected orientation of $K_{r,t}$ for $r$ sufficiently large, as shown by a theorem of Kollár, Rónyai and Szabó \cite{krs}. However, it is not sharp for every graph satisfying the assumption, since for example the path is such a graph.



\section{Concluding remarks}

Let us recall that $\exd(n,\overrightarrow{F})$ has attracted more attention than $\exo(n,\overrightarrow{F})$  so far. Our main result, Theorem \ref{main} is equivalent to the statement $\exd(n,H)\le c\cdot n^{2-\frac{1}{r}}$, thus it gives a new result in that setting as well.

Our second main contribution is the stability method in this setting. This does not extend completely to directed graphs. However, in the case the forbidden directed graph is an orientation $\overrightarrow{F}$ of a bipartite $F$, then it does. We know that the underlying graph of an $\overrightarrow{F}$-free oriented $n$-vertex graph with at least $\left(\frac{z(\overrightarrow{F})-2}{z(\overrightarrow{F})-1}-o(1)\right)\binom{n}{2}$ edges can be turned to $T(n,z(\overrightarrow{F}-1))$ by adding and deleting $o(n^2)$ edges, but this does not say anything about the largest $\overrightarrow{F}$-free directed $n$-vertex graph, since that may contain edges directed both directions. However, the subgraph consisting of such edges must be $F$-free. Therefore, by deleting $o(n^2)$ edges we obtain an oriented graph, and then by adding and deleting $o(n^2)$ edges, we obtain  $T(n,z(\overrightarrow{F})-1)$.

\vskip 1cm

\noindent {\bf Acknowledgement.} We are thankful to Professor Hong Liu for discussions on the undirected version of random zooming theorem.

\end{document}